\documentclass[12pt,reqno]{amsart}
\usepackage{diagrams}
\usepackage{amssymb}
\usepackage{graphicx}
\usepackage{amsmath}
\usepackage{a4wide}
\usepackage{version}
\usepackage{mathrsfs}
\usepackage{tikz}
\usepackage{mathtools}
\usetikzlibrary{arrows,decorations.markings, matrix}
\usepackage{tikz-cd}
\usepackage{hyperref}
\usepackage{float}
\usepackage[all]{xy}

\numberwithin{equation}{section}

\newtheorem{thm}{Theorem}[subsection]
\newtheorem{lem}[thm]{Lemma}

\theoremstyle{definition}

\newtheorem{rmk}[thm]{Remark}

\newcommand{\Z}{\mathbb{Z}}

\newcommand{\id}{\operatorname{id}}

\newcommand{\we}{\wedge}

\newcommand{\rk}{\operatorname{rk}}

\newcommand{\Om}{\Omega}

\newcommand{\Hom}{\operatorname{Hom}}

\newcommand{\Ext}{\operatorname{Ext}}

\newcommand{\Res}{\operatorname{Res}}

\renewcommand{\a}{\alpha}
\renewcommand{\b}{\beta}
\newcommand{\om}{\omega}
\newcommand{\De}{\Delta}

\renewcommand{\P}{{\mathbb P}}

\newcommand{\wt}{\widetilde}
\newcommand{\ot}{\otimes}

\newcommand{\sub}{\subset}

\renewcommand{\mod}{\operatorname{mod}}

\newcommand{\OO}{{\mathcal O}}

\newcommand{\DD}{{\mathcal D}}

\newcommand{\SL}{\operatorname{SL}}

\newcommand{\G}{{\mathbb G}}

\newcommand{\lan}{\langle}
\newcommand{\ran}{\rangle}

\renewcommand{\P}{{\mathbb P}}

\newcommand{\si}{\sigma}

\newcommand{\ga}{\gamma}

\renewcommand{\ker}{\operatorname{ker}}
\newcommand{\im}{\operatorname{im}}

\newcommand{\Sec}{\operatorname{Sec}}
\newcommand{\pa}{\partial}

\title[]{Feigin-Odesskii brackets, syzygies, and Cremona transformations}
\author{Alexander Polishchuk}
\thanks{Supported in part by the NSF grant DMS-2001224, 
and within the framework of the HSE University Basic Research Program and by the Russian Academic Excellence Project `5-100'.}

\address{University of Oregon, National Research University Higher School of Economics, and
Korea Institute for Advanced Study}

\begin{document}

\begin{abstract}
We identify Feigin-Odesskii brackets $q_{n,1}(C)$, associated with a normal elliptic curve of degree $n$, $C\sub \P^{n-1}$,
with the skew-symmetric $n\times n$ matrix of quadratic forms introduced by Fisher in \cite{Fisher-jac} in connection with
some minimal free resolutions related to the secant varieties of $C$. On the other hand, we show that for odd $n$, the generators of
the ideal of the secant variety of $C$ of codimension $3$ give a Cremona transformation of $\P^{n-1}$,
generalizing the quadro-cubic Cremona transformation of $\P^4$. 
We identify this transformation with the one considered in \cite{P98} and find explict formulas for the inverse transformation.
We also find polynomial formulas for Cremona transformations from \cite{P98} associated with higher rank bundles on $C$.

\end{abstract}

\maketitle

\section{Introduction}

Let $C\sub \P^{n-1}=\P(V)$ be an elliptic normal curve of degree $n\ge 3$ over a field $k$ of characteristic $0$. 
With $C$ one can associate canonically (up to rescaling) a Poisson bracket $q_{n,1}(C)$ on the ambient projective space $\P(V)$.
Namely, if $L=\OO(1)|_C$, a line bundle of degree $n$ on $C$, then we can identify $\P(V)$ with
$$\P H^0(C,L)^*\simeq \P \Ext^1(L,\OO)$$
(where we used Serre duality on $C$).
Now there is a natural Poisson bracket on the projective space $\P \Ext^1(L,\OO)$ described in \cite{FO95} and \cite{P98}.

The Poisson brackets $q_{n,1}(C)$ correspond to quadratic brackets on $V$ arising as semi-classical limit of Feigin-Odesskii elliptic
algebras (see \cite{FO89}). The Poisson geometry of $q_{n,1}(C)$ has been well studied (see e.g., \cite{FO95}, \cite{P95}, \cite{P98}, \cite{O-bih}, \cite{OW},
\cite{PS}, \cite{HP-mod}, \cite{HP-bih}). In particular, it is known that for odd $n$
the stratification of $\P^{n-1}$ by the rank of this Poisson bracket is related to the secant varieties of $C$.

The first goal of the present work is to establish a direct relation of $q_{n,1}(C)$ with a certain skew-symmetric matrix of quadratic forms $\Om$ related to 
the secant varieties of $C$, introduced and studied by Fisher in \cite{Fisher-jac}.

Assume first that $n$ is odd and let $r=(n-1)/2$. Then as is well known the secant variety $\Sec^rC$ is a hypersurface of degree $n$ in $\P^{n-1}$.
Let $F(x_1,\ldots,x_n)=0$ be its equation (of degree $n$). Then Fisher showed in \cite{Fisher-jac} that the module of syzygies between the partial derivatives 
$(\frac{\pa F}{\pa x_1}, \ldots, \frac{\pa F}{\pa x_n})$ is generated by the relations 
\begin{equation}\label{FOm-id}
\sum_i \frac{\pa F}{\pa x_i}\Om_{ij}=0,
\end{equation}
for a unique (up to rescaling) skew-symmetric 
$n\times n$ matrix of quadratic forms $\Om$ (see Sec.\ \ref{sec-sec} for details).

In the case when $n$ is even, we set $r=(n-2)/2$. Then the secant variety $\Sec^r C$ is a complete intersection of codimension $2$ given by $F_1=F_2=0$, where 
$\deg(F_1)=\deg(F_2)=r+1$. In this case the $n\times n$ skew-symmetric matrix of quadratic forms $\Om$ is characterized by the condition that the relations
\begin{equation}\label{F12Om-id}
\sum_i \frac{\pa F_a}{\pa x_i}\Om_{ij} =0, \text{ where } a=1,2,
\end{equation} 
generate the module of syzygies between the columns of the $2\times n$ matrix
$(\frac{\pa F_a}{\pa x_j})$.

\medskip

\noindent
{\bf Theorem A}. {\it (i) The formula $\{x_i,x_j\}=\Om_{ij}$ defines a Poisson bracket on $V$ such that the induced Poisson bracket on $\P(V)$ is $q_{n,1}(C)$ (up to rescaling).

\noindent
(ii) If $n$ is odd then $\{x_i,F\}=0$ for every $i$. If $n$ is even then $\{x_i,F_1\}=\{x_i,F_2\}=0$ for every $i$.
}

\medskip

Note that part (ii) is stated in \cite[Sec.\ 5]{O-ell} (see also \cite{O-bih}).


Assume now that $n$ is odd and $n\ge 5$.
Our second result gives an explicit formula for the Poisson birational transformation from $(\P^{n-1},q_{n,1}(C))$ to $(\P^{n-1},q_{n,(n-1)/2}(C))$ constructed in \cite{P98}.
Recall that for any relatively prime $(n,k)$, the Feigin-Odesskii bracket $q_{n,k}(C)$ is a natural Poisson bracket on $\P^{n-1}=\P \Ext^1(V,\OO)$, where $V$ is a stable bundle
of rank $k$ and degree $n$ on $C$ (see \cite{FO95}, \cite{P98}). 

Let us start with a line bundle $L=\OO(1)|_C$ of degree $n$, and let $V_0$ be the unique stable bundle of rank $2$ with $\det(V_0)=L$
(it exists since $n$ is odd). In \cite{P98} we constructed a natural birational map (a Cremona transformation)
$$\phi:\P^{n-1}=\P\Ext^1(\OO,L)\dashrightarrow \P H^0(C,V_0)=\P^{n-1}$$ 
by observing that for a generic extension
$$0\to \OO\to E\to L\to 0$$
the bundle $E$ is stable, so there is an isomorphism $E\simeq V_0$, unique up to rescaling. Hence, from the above extension we get a nonzero section of $V_0$,
well defined up to rescaling. Furthermore, we showed in \cite{P98} that $\phi$ is compatible with Poisson structures, where $\P H^0(C,V_0)$ is equipped with a natural Poisson bracket $q$,
such that there exists a Poisson isomorphism
$$(\P H^0(C,V_0),q)\simeq (\P \Ext^1(V,\OO),q_{n,r}(C)),$$
where $V$ is a stable bundle of degree $n$ and rank $r=(n-1)/2$.

Now let us look at the variety $\Sec^{r-1} C$ of codimension $3$ in $\P^{n-1}$. It is known that its homogeneous ideal is generated by $n$ forms $(p_1,\ldots,p_n)$ of degree $r$.

\medskip

\noindent
{\bf Theorem B}. {The above birational map $\phi:\P^{n-1}\dashrightarrow \P^{n-1}$ is given by 
$$(x_1:\ldots:x_n)\mapsto (p_1(x):\ldots:p_n(x)).$$ 
The inverse is given by $(y_1:\ldots:y_n)\mapsto (f_1(y):\ldots:f_n(y))$, where $f_i$ are certain homogeneous polynomials of degree $n-2$.

\medskip

For example, when $n=5$, we have $r=2$ and $(p_1,\ldots,p_5)$ are the quadrics generating the ideal of $C$ in $\P^4$. The corresponding Cremona transformation of $\P^4$ is
well known classically as quadro-cubic Cremona transformation (see \cite{Semple}).
The polynomials $(p_1,\ldots,p_n)$ can be computed as (signed) submaximal pfaffians of a skew-symmetric $n\times n$ matrix of linear forms $\Phi$, called the Klein matrix in \cite{Fisher-pf}.
For a Heisenberg invariant curve $C$, there is a simple formula for $\Phi$ (see \cite[Prop.\ 3.7]{Fisher-pf}).
In the course of the proof of Theorem B we also give a recipe for computing the polynomials $f_i$ in terms of $\Phi$.

\medskip

\noindent
{\bf Corollary C}. {\it For $r\ge 2$, any hypersurface of degree $r$ in $\P^{2r}$ containing $\Sec^{r-1} C$, where $C\sub \P^{2r}$ is an elliptic normal curve, is rational.}

\medskip

Indeed, the corresponding birational map $\phi$ induces a birational map from such a hypersurface to a hyperplane in $\P^{2r}$.

Finally, we find polynomial formulas for similar Cremona transformations associated with stable bundles of higher rank on $C$ (these transformations were defined
in \cite{P98}). Namely,
for $k>0$ relatively prime to $n$, let $V_k$ be the unique stable vector bundle of rank $k$ on $C$ with $\det V_k=L$ (where $\deg L=n$).
Assume that $k+1$ is still relatively prime to $n$. Then we can also consider the bundle $V_{k+1}$ of rank $k+1$.
Now all three spaces $H^0(C,V_{k+1})=\Hom(\OO,V_{k+1})$, $\Hom(V_{k+1},V_k)$ and $\Ext^1(V_k,\OO)$ have dimension $n$.
As in \cite{P98}, let us consider the closed subvariety
$$Z\sub \P \Hom(\OO,V_{k+1})\times \P \Hom(V_{k+1},V_k)\times \P \Ext^1(V_k,\OO)$$
which is the closure of the locus of $(\a,\b,\ga)$ such that 
$$\OO\rTo{\a} V_{k+1}\rTo{\b} V_k \rTo{\ga} \OO[1]$$
is an exact triangle. Then the projection from $Z$ to any of the factors is birational, so we get three Cremona transormations
\begin{align*}
& \phi:\P \Ext^1(V_k,\OO)\dashrightarrow \P \Hom(V_{k+1},V_k), \ \ \psi:\P \Hom(V_{k+1},V_k)\dashrightarrow \P H^0(C,V_{k+1}), \\ 
& \rho:\P H^0(C,V_{k+1})\dashrightarrow \P \Ext^1(V_k,\OO),
\end{align*}
such that $\rho\psi\phi=\id$. 

\medskip

\noindent
{\bf Theorem D}. {\it Let us define numbers $l$ and $m$, where $0<l<n$, $0<m<n$, by
$$l\equiv -(k+1)^{-1} \mod (n), \ \ m\equiv -1-k^{-1} \mod (n).$$
Then there exist linear maps
\begin{align*}
& \wt{\phi}:S^l(\Ext^1(V_k,\OO))\to \Hom(V_{k+1},V_k), \ \ \wt{\psi}:S^k(\Hom(V_{k+1},V_k))\to H^0(C,V_{k+1}),\\
& \wt{\rho}:S^m(H^0(C,V_{k+1}))\to \Ext^1(V_k,\OO),
\end{align*}
inducing $\phi$, $\psi$ and $\rho$. 
}

\medskip

We give an explicit construction for the polynomial map $\wt{\psi}$, and then we construct $\wt{\phi}$ and $\wt{\rho}$ using autoequivalences of the derived category of $C$.
Note that in the case $k=1$, the map $\wt{\psi}$ is a linear isomorphism, while the polynomials giving $\wt{\phi}$ and $\wt{\rho}$ have the same degrees as the polynomial maps of Theorem B.
In fact, it is easy to see that in this case $\wt{\phi}$ coincides with $(p_1,\ldots,p_n)$, and we expect $\wt{\rho}$ to coincide with $(f_1,\ldots,f_n)$ constructed in the proof of Theorem B
(see Remark \ref{two-map-rem}).

\bigskip

\noindent
{\it Acknowledgment}. This paper was inspired by a discussion with Volodya Rubtsov in 2017 and by his paper \cite{Rubtsov}, 
where the Feigin-Odesskii brackets $q_{5,1}(C)$ and the corresponding quadro-cubic Cremona transformations of $\P^4$ are studied. 
I am very grateful to Volodya for sharing his insights on Feigin-Odesskii brackets and for drawing my attention to the works \cite{Fisher-pf}, \cite{Fisher-jac} on
secant varieties of normal elliptic curves. 

\section{Preliminaries}

\subsection{Quadratic Poisson brackets}

Let $V$ be a vector space of dimension $n$. By a quadratic Poisson bracket on $V$ we mean a Poisson structure on the algebra $S(V^*)$ compatible with the grading, i.e.,
such that the Poisson brackets of linear forms are given by homogeneous quadratic forms. Equivalently, the corresponding bivector on $V$ has to be homogeneous (i.e., preserved
by the natural $\G_m$-action).
It is well known that every such bracket induces a Poisson bracket on
the projective space $\P(V)$ and that every Poisson bracket on $\P(V)$ comes from a quadratic Poisson bracket on $V$, not necessarily unique (see \cite{B}, \cite{P95}).

Different liftings of a bivector on $\P(V)$ to a homogeneous bivector on $V$ differ by bivectors on $V$ of the form $\lan x,y\ran=A(x)y-A(y)x$, where $x,y\in V^*$,
for some linear operator $A:V^*\to V^*$. This easily implies that different liftings of a Poisson bracket $\Pi$ on $\P(V)$ to quadratic Poisson brackets on $V$ are numbered by
Poisson vector fields for $\Pi$ on $\P(V)$, i.e., global vector fields on $\P(V)$ preserving $\Pi$.

The following simple criterion is helpful in identifying a lifting of a Poisson bracket on $\P(V)$ to a quadratic Poisson bracket on $V$.

\begin{lem}\label{Cas-lem}
Let $\{\cdot, \cdot\}$ be a homogeneous bivector on $V$, inducing a Poisson bracket on $\P(V)$. Assume that there exists a nonzero homogeneous polynomial $F\in S^dV^*$ of
positive degree $d$ such that $\{x,F\}=0$ for every $x\in V^*$. 
Then $\{\cdot,\cdot\}$ is a Poisson bracket, i.e., it satisfies the Jacobi identity.
\end{lem}

\begin{proof}
Set $J(x,y,z):=\{\{x,y\},z\}+\{\{y,z\},x\}+\{\{z,x\},y\}$. First, let us check that $J(x,y,z)=0$ for all $x,y,z\in S^dV^*$.
Indeed, $x/F$, $y/F$ and $z/F$ are local functions on $\P(V)$, so 
we know that
$$J(\frac{x}{F},\frac{y}{F},\frac{z}{F})=0.$$
Furthermore, since $\{x,F\}=\{y,F\}=\{z,F\}=0$, we get $\{\frac{x}{F},\frac{y}{F}\}=\frac{\{x,y\}}{F^2}$, etc., hence,
$$J(\frac{x}{F},\frac{y}{F},\frac{z}{F})=\frac{J(x,y,z)}{F^3}.$$
Thus, we get that $J(x,y,z)=0$ for all $x,y,z\in S^dV^*$.

Hence, if $x,y,z$ are in $V^*$ then we have 
$J(x^d,x^{d-1}y,x^{d-1}z)=0$. Now the vanishing of $J(x,y,z)$ follows from the formal identity
$$J(x^d,x^{d-1}y,x^{d-1}z)=dx^{3d-3}J(x,y,z).$$
\end{proof}

\subsection{Formula for Feigin-Odesskii bracket}

Let us recall the formula for the Poisson bivector $\Pi=q_{n,1}(C)$ on $\P \Ext^1(L,\OO)\simeq \P H^0(C,L)^*$ in terms of Szeg\"o kernel, established in \cite{HP-bih}
(equivalent formulas are also stated in \cite{O-bih} and in \cite{OW}).

Let $p\in C$ be a point, and let us fix a trivialization of $\om_C$.
The Szeg\"o kernel is the unique section $S\in H^0(C\times C,\OO(p)\boxtimes \OO(p)(\De))$ such that $\Res_{\De}(S)=1$ and $S(y,x)=-S(x,y)$.
If $C$ is identified with a plane cubic $y^2=P(x)$ and the trivializing global differential is $dx/2y$, then one has
$$S=\frac{y_1+y_2}{x_2-x_1}$$
(see \cite[Sec.\ 5.1.2]{HP-bih}).

To give a Poisson bivector $\Pi$ we need to specify for each $\phi\in H^0(L)^*$ a skew-symmetric form $\Pi_\phi$ on the cotangent space 
$$T^*_\phi \P H^0(L)^*\simeq \ker(\phi)\sub H^0(L).$$
This form is given by the formula (see \cite[Lem.\ 2.1, Prop.\ 5.8]{HP-bih})
\begin{equation}\label{Pi-phi-formula}
\Pi_\phi(s_1\we s_2)=\pm \lan \wt{\phi}\ot \wt{\phi}, S\cdot (s_1\boxtimes s_2-s_2\boxtimes s_1)\ran,
\end{equation}
for $s_1,s_2\in \ker(\phi)$.
Here we consider $H^0(L)$ as a subspace in $H^0(L(p))$ and denote by $\wt{\phi}$ any extension of $\phi$ to a functional on $H^0(L(p))$.
Further, we view $s_1\boxtimes s_2-s_2\boxtimes s_1$ as a section of $L\boxtimes L$ on $C\times C$, vanishing on the diagonal.
Hence, the product $S\cdot (s_1\boxtimes s_2-s_2\boxtimes s_1)$ is a global section of $L(p)\boxtimes L(p)$.

The fact that the right-hand side of \eqref{Pi-phi-formula} does not depend on a choice of $\wt{\phi}$ follows from the existence of a linear operator
$D:H^0(L)\to H^0(L(p))$ such that for any $s_1,s_2\in H^0(L)$ one has
$$S\cdot (s_1\boxtimes s_2-s_2\boxtimes s_1)+D\{s_1,s_2\}\in H^0(L)\ot H^0(L), \text{ where}$$
$$D\{s_1,s_2\}:=s_1\boxtimes D(s_2)+D(s_2)\boxtimes s_1-s_2\boxtimes D(s_1)-D(s_1)\boxtimes s_2$$
(so $D$ cancels the poles at $p$). The existence of such $D$ was first observed in \cite{O-bih} (see also \cite{OW} and \cite[Sec.\ 5.3]{HP-bih}).
For $s_1,s_2\in \ker(\phi)$ this allows to rewrite the formula \eqref{Pi-phi-formula} as
\begin{equation}\label{Pi-phi-formula-bis}
\Pi_\phi(s_1\we s_2)=\pm \lan \phi\ot \phi, S\cdot (s_1\boxtimes s_2-s_2\boxtimes s_1+D\{s_1,s_2\})\ran.
\end{equation}

\subsection{Secant varieties and syzygies}\label{sec-sec}

Let us recall some results of Fisher in \cite{Fisher-pf} and \cite{Fisher-jac}  on secant varieties of a normal elliptic curve $C\sub \P^{n-1}$.

First, assume that $n$ is odd, and set $r=(n-1)/2$. 
The variety $\Sec^r C$ is a hypersurface in $\P^{n-1}$ given by $F=0$, where $\deg(F)=n$.
The first result, \cite[Thm.\ 1.1(i)]{Fisher-jac} gives a form of the minimal free resolution of the ideal $(\frac{\pa F}{\pa x_1}, \ldots, \frac{\pa F}{\pa x_n})$. 
Namely, there is a unique
skew-symmetric $n\times n$ matrix $\Om$ of quadratic forms, such that the following complex is exact, where $R=k[x_1,\ldots,x_n]$:
\begin{equation}\label{Om-complex-odd}
0\to R(-2n)\rTo{\nabla^T} R(-n-1)^n\rTo{\Om} R(-n+1)^n\rTo{\nabla}R,
\end{equation}
where $\nabla=(\frac{\pa F}{\pa x_1}, \ldots, \frac{\pa F}{\pa x_n})$. 
Note that the fact that this is a complex corresponds to the identity \eqref{FOm-id}.

In the case when $n$ is even, let us set $r=(n-2)/2$. Then the variety $\Sec^r C$ is a complete intersection of codimension $2$ given by $F_1=F_2=0$, where $\deg(F_i)=r+1$
(see \cite[Sec.\ 8]{vBH}, \cite[Thm.\ 9.1]{Fisher-jac}). Now setting
$$\nabla=\left(\begin{matrix} \frac{\pa F_1}{\pa x_1} & \ldots & \frac{\pa F_1}{\pa x_n} \\ \frac{\pa F_2}{\pa x_1} & \ldots & \frac{\pa F_2}{\pa x_n} \end{matrix}\right),$$
the statement of  \cite[Thm.\ 1.1(ii)]{Fisher-jac} is that there  is a unique
skew-symmetric $n\times n$ matrix $\Om$ of quadratic forms, such that the following complex is exact:
\begin{equation}\label{Om-complex-even}
0\to R(-n)^2\rTo{\nabla^T} R(-r-1)^n\rTo{\Om} R(-r)^n\rTo{\nabla}R^2.
\end{equation}
Again, the fact that this is a complex corresponds to the identities \eqref{F12Om-id}.
 
Next, in the case when $n$ is odd, we will use the following description of the minimal free resolution of the ideal of $\Sec^{r-1} C$ in $R[x_1,\ldots,x_n]$, where $r=(n-1)/2$
(see \cite[Sec.\ 4]{Fisher-pf}). This ideal is generated by $n$ linearly independent forms of degree $r$, $(p_1,\ldots,p_n)$, and there exists a unique skew-symmetric $n\times n$ matrix
of linear forms $\Phi$ such that the following complex is exact:
\begin{equation}\label{Klein-complex}
0\to R(-n)\rTo{p^T} R(-r-1)^n\rTo{\Phi} R(-r)^n\rTo{p} R,
\end{equation}
where $p=(p_1,\ldots,p_n)$.

Furthermore, we will use the following description of $\Phi$, which is called {\it Klein matrix} in \cite{Fisher-pf}. 
Let $V_0$ be the unique stable bundle on $C$ with $\det(V_0)\simeq L=\OO(1)|_C$. Then identifying the space of linear 
forms on $\P^{n-1}$ with $H^0(C,L)$, the skew-symmetric matrix $\Phi$
corresponds to the natural $H^0(C,L)$-valued skew-symmetric pairing on $H^0(C,V_0)$,
\begin{equation}\label{Klein-matrix}
\Phi:{\bigwedge}^2H^0(C,V_0)\to H^0(C,\det(V_0))\simeq H^0(C,L).
\end{equation}

\begin{rmk}
Both complexes \eqref{Om-complex-odd} and \eqref{Klein-complex} for odd $n$ are examples of Buchsbaum-Eisenbud pfaffian presentations of the corresponding Gorenstein ideals of
height $3$ (see \cite{BE}). In particular, the generators of the ideal ($\frac{\partial F}{\partial x_i}$ in the former case and $p_i$ in the latter case), are equal to the (signed) submaximal pfaffians
of the corresponding skew-symmetric matrix ($\Om$ in the former case and $\Phi$ in the latter case).
\end{rmk}

\begin{rmk} The fact that the submaximal pfaffians $(p_i)$ vanish on $\Sec^{r-1}C\sub \P H^0(C,L)^*$ means that for $\xi\in \Sec^{r-1}C$ the skew-symmetric form $\Phi_\xi$ on $H^0(C,V_0)$
has rank $\le n-3$. This can be explained geometrically as follows. A point $\xi\in \Sec^{r-1}C$ corresponds to a functional that factors as $H^0(C,L)\to H^0(C,L|_D)\to k$,
where $D$ is an effective divisor of degree $r-1$ on $C$. But then the $3$-dimensional subspace $H^0(C,V_0(-D))\sub H^0(C,V_0)$ is in the kernel of $\Phi_\xi$.
\end{rmk}

\subsection{Skew-symmetric matrices of linear forms}

Let $A$ and $B$ be a pair of $n$-dimensional vector spaces, and let $\Phi\in {\bigwedge}^2A^*\ot B$.
Typically one views $\Phi$ as a skew-symmetric matrix of linear forms on $B^*$. However, $\Phi$ also gives
a linear map $\nu_{\Phi}:A\to B\ot A^*$, so we can view it as an $n\times n$ matrix of linear forms on $A$.

More precisely, consider the polynomial algebra $R=S^*(A^*)$. Then we can view $\nu_{\Phi}$ as an $R$-valued linear map
$A\to B\ot R$, and consider the corresponding matrix of minors,
$${\bigwedge}^{n-1}(\nu_{\Phi}):{\bigwedge}^{n-1}A\to {\bigwedge}^{n-1}B\ot R$$
with entries in $S^{n-1}(A^*)\sub R$.

\begin{lem}\label{si-Phi-lem}
(i) Consider the composed map
\begin{align*}
&\si(\Phi):A^*\rTo{\sim} \det(A)^*\ot {\bigwedge}^{n-1}(A)\rTo{\id\ot {\bigwedge}^{n-1}(\nu_{\Phi})} \det(A)^*\ot {\bigwedge}^{n-1}(B)\ot R\\
&\simeq \det(A)^*\ot \det(B)\ot B^*\ot R.
\end{align*}
Then for every $\xi\in A^*$, the element $\si(\Phi)(\xi)$ (of the target $R$-module) is divisible by $\xi$.

\noindent
(ii) There exists a unique element
$$\si_{n-2}(\Phi)\in \det(A)^*\ot \det(B)\ot B^*\ot S^{n-2}(A^*),$$
such that 
$$\si(\Phi)(\xi)=\si_{n-2}(\Phi)\cdot \xi$$
(where the product is in $R$).

\noindent
(iii) Viewing $\si_{n-2}(\Phi)$ as a $\det(A)^*\ot\det(B)\ot R$-valued linear form on $B$, we have 
$$\si_{n-2}(\Phi)\circ \nu_\Phi=0$$
as a $\det(A)^*\ot\det(B)\ot R$-valued linear form on $A$.

\noindent
(iv) Assume that $\Phi(a,?):A\to B$ has rank $\ge n-1$ for some $a\in A$. Then $\si_{n-2}(\Phi)(a)\neq 0$.
\end{lem}

\begin{proof}
(i) Let us think of $\nu_\Phi$ as a $\Hom(A,B)$-valued function on $A$. Then $\si(\Phi)(\xi)$ is given by the $(n-1)\times(n-1)$-minors
of the corresponding $\Hom(\lan \xi\ran^\perp,B)$-valued function on $A$, 
$$\nu_{\Phi,\xi}:\lan \xi\ran^\perp\to B\ot A^*,$$
where we use the restriction from $A$ to the hyperplane $\lan \xi\ran^\perp\sub A$.
The condition that all $n\times n$-minors of $\nu_{\Phi,\xi}$ are divisible by $\xi$ is equivalent to the condition that for every $a\in \lan \xi\ran^\perp$,
the induced linear map
$$\nu_{\Phi,\xi,a}:\lan \xi\ran^\perp\to B$$
has rank $<n-1$. It is enough to check that for $a\neq 0$, the map $\nu_{\Phi,\xi,a}$ has nonzero kernel.
But this follows from the identity
$$\nu_{\Phi,\xi,a}(a)=0$$
which holds by the skew-symmetry of $\Phi$.

\noindent
(ii) 
Note that the target of $\si(\Phi)$ is a free $R$-module. Thus, it is enough to prove that if a linear map
$\si:A^*\to R=S(A^*)$ has the property that $\si(\xi)$ is divisible by $\xi$ for every $\xi\in A^*$, then $\si(\xi)=f\cdot \xi$ for a fixed polynomial
$f\in R$. 

Indeed, assume that for every $\xi\neq 0$, we have 
$$\si(\xi)=f_{\xi}\cdot \xi$$
for some $f_\xi\in R$. We need to check that $f_{\xi_1}=f_{\xi_2}$ for a linearly independent pair $\xi_1,\xi_2\in A^*$.
For any $c\in k^*$ we have
$$\si(\xi_1+c\xi_2)=f_{\xi_1}\xi_1+f_{\xi_2}\xi_2=f_c(\xi_1+c\xi_2),$$
where $f_c:=f_{\xi_1+c\xi_2}$. This implies that $(f_c-f_{\xi_1})\xi_1$ is divisible by $\xi_2$, so we can write
$$f_c=f_{\xi_1}+f \xi_2.$$
From this we get
$$f_{\xi_2}-f_{\xi_1}=(c^{-1}x_1+x_2)\cdot f,$$
so $f_{\xi_2}-f_{\xi_1}$ is divisible by $c^{-1}x_1+x_2$. Since $k$ is infinite, this implies that $f_{\xi_2}-f_{\xi_1}=0$.

\noindent
(iii) We use again the fact that for a nonzero $\xi\in A^*$, $\si(\Phi)(\xi)$, which is an $R$-valued functional on $B$, is given by ${\bigwedge}^{n-1}(\Phi|_{\lan\xi\ran^\perp})$,
where we consider the restriction 
$$\Phi|_{\lan \xi\ran^\perp}:\lan\xi\ran^\perp\to B\ot R,$$
where $\dim \lan\xi\ran^\perp=n-1$. It follows that the composition
$$\lan\xi\ran^\perp\rTo{\Phi} B\ot R\rTo{\si(\Phi)(\xi)} \det(A)^*\ot \det(B)\ot R$$
is zero. By part (ii) this implies that the composition 
$$\si_{n-2}(\Phi)\circ \Phi:A\to \det(A)^*\ot \det(B)\ot R$$
has zero restriction to $\lan \xi\ran^\perp$ for every $\xi\neq 0$. Hence, this composition is zero.

\noindent
(iv) The assumption implies that $\si(\Phi)(a)\neq 0$. Now the assertion follows from (ii).
\end{proof}

\section{Proofs}

\subsection{Proof of Theorem A}

Working over an algebraic closure of $k$, we can assume that $L=\OO(np)$ for a point $p\in C$.
In this case Fisher gives the following explicit description of the matrix 
$$\Om:{\bigwedge}^2 H^0(C,L)\to S^2 H^0(C,L)$$ 
in \cite[Sec.\ 5]{Fisher-jac}. For $s_1,s_2\in H^0(C,L)$ and $\phi\in H^0(C,L)^*$, one has
$$\frac{1}{n}\Om(s_1,s_2)(\phi)=[S\cdot (s_1\boxtimes s_2-s_2\boxtimes s_1)](\phi,\phi)+\DD\{s_1,s_2\}(\phi,\phi),$$
for some operator $\DD:H^0(C,\OO(np))\to H^0(C,\OO((n+1)p))$ (which is given up to rescaling by $f\mapsto df/\om$, where $\om$ is a global differential on $C$).

Comparing the above formula for $\Om$ with \eqref{Pi-phi-formula-bis} we immediately deduce that the bivector on $\P^{n-1}$ induced by $\Om$ is $q_{n,1}(C)$, up to rescaling.

In the case $n$ is odd we have
$$\{F,x_j\}=\sum_i \frac{\pa F}{\pa x_i}\{x_i,x_j\}=\sum_i \frac{\pa F}{\pa x_i}\Om_{ij}=0$$
by \eqref{FOm-id}. Hence, the assertion of the theorem follows in this case from Lemma \ref{Cas-lem}.

Similarly, in the case $n$ is even, we deduce from \eqref{F12Om-id} that $\{F_1,x_j\}=\{F_2,x_j\}=0$.






\subsection{Proof of Theorem B}

We start with the following description of
$$\phi^{-1}:\P H^0(C,V_0)\dashrightarrow \P\Ext^1(\OO,L)\simeq \P H^0(C,L)^*.$$ 
Given a generic global section $s\in H^0(C,V_0)$, the corresponding map $s:\OO\to V_0$ is an embedding of a subbundle, and we have a canonical
identification of $V_0/s(\OO)$ with $\det(V_0)=L$. Hence, we get an extension of $L$ by $\OO$.
Note that the Serre duality isomorphism $\Ext^1(\OO,L)\simeq H^0(C,L)^*$ associates with an extension
$$0\to \OO\to E\to L\to 0$$
the corresponding coboundary homomorphism $H^0(C,L)\to H^1(C,\OO)\simeq k$.
The exact sequence of cohomology shows that the kernel of this homomorphism coincides with the image of the map $H^0(C,E)\to H^0(C,L)$.

Thus, $\phi^{-1}$ associates with a generic $s\in H^0(C,V_0)$ the unique functional $\xi\in H^0(C,L)^*$ (up to rescaling), such that $\ker(\xi)$ is equal to the image of
the map
$$d_s:H^0(C,V_0)\to H^0(C,V_0/s(\OO))\simeq H^0(C,L).$$
Recall that an isomorphism $\a_s:V_0/s(\OO)\rTo{\sim} L$
induces an isomorphism 
$$\b_s:\det(V_0)\rTo{\sim} L: s\we x\mapsto \a(x).$$ 
Note that $\b_s$ does not depend on $s$ up to rescaling. 

Thus, up to rescaling, the map $d_s$ can be identified with the map 
$$\Phi(s,?):H^0(C,V_0)\to H^0(C,L),$$ 
where $\Phi$ is the Klein matrix \eqref{Klein-matrix}.
In other words, $\xi=\phi^{-1}(s)$ is characterized by the condition 
$$\ker(\xi)=\im \Phi(s,?).$$

Let us set $A:=H^0(C,V_0)$ and $B:=H^0(C,L)$ for brevity (note that $B$ is the space of linear forms on our projective space $\P^{n-1}=\P H^0(C,L)^*$), so that $\Phi$
can be viewed as a skew-symmetric linear map
$$\Phi:A\to A^*\ot B.$$ 
We showed above that the birational map 
$$\phi^{-1}:\P A\dashrightarrow \P B^*$$
sends a generic $a$ to the unique functional $b^*$ (up to rescaling) such that $b^*$ vanishes on the image of $\Phi(a,?):A\to B$.
Now we recall that with $\Phi$ we can associate canonically (up to rescaling) an element
$\si_{n-2}(\Phi)\in B^*\ot S^{n-2}(A^*)$, such that the composition
$$A\rTo{\Phi(a,?)}B\rTo{\si_{n-2}(\Phi)(a)}k$$
is zero (see Lemma \ref{si-Phi-lem}). Furthermore, since $\Phi(a,?)$ has rank $n-1$ for generic $a$, we have
$\si_{n-2}(\Phi)(a)\neq 0$ for such $a$ (see Lemma \ref{si-Phi-lem}(iv)).
This implies that
$$\phi^{-1}(a)=\si_{n-2}(\Phi)(a)$$
in $\P B^*$.
If we choose a basis in $B$, then $\si_{n-2}(\Phi)$ is given by $n$ polynomials $(f_1,\ldots,f_n)$, where $f_i\in S^{n-2}(A^*)$, and we have
$$\phi^{-1}(a)=(f_1(a):\ldots:f_n(a)).$$

It remains to prove the formula for $\phi:\P B^*\to \P A$.
We can rewrite the complex \eqref{Klein-complex} as
$$0\to R(-n)\rTo{p^T} A\ot R(-r-1)\rTo{\Phi} A^*\ot R(-r)\rTo{p} R$$
where $R=S^*(B)$.
This complex shows that for each $b^*\in B^*$ the specialization $\Phi_{b^*}:A\to A^*$ satisfies
\begin{equation}\label{Phi-b*-p-eq}
\Phi_{b^*}(p(b^*)^T)=0,
\end{equation}
where $p(b^*)^T\in A$. 

We need to check that $\phi(b^*)=p(b^*)^T$. Equivalently, setting 
$$a:=p(b^*)^T,$$ 
we need to prove that 
$$\phi^{-1}(a)=b^*.$$
By the above description of $\phi^{-1}$, it is enough to check that $b^*$ annihilates the image of $\Phi(a,?):A\to B$ (since the latter image is a hyperplane in $B$).
In other words, we need to check that $b^*\circ \Phi(a,?)=0$. But we have
$$b^*\circ \Phi(a,?)=\Phi_{b^*}(a)=0$$
by \eqref{Phi-b*-p-eq}. This ends the proof.

\subsection{Proof of Theorem D}

First, let us construct $\wt{\psi}:S^k(\Hom(V_{k+1},V_k))\to H^0(C,V_{k+1})$.
A generic morphism $f:V_{k+1}\to V_k$ fits into an exact sequence
$$0\to \OO\rTo{s} V_{k+1} \rTo{f} V_k\to 0$$
and our goal is to find a degree $k$ polynomial formula for $s\in H^0(C,V_{k+1})$ in terms of $f$.
Let us consider the induced map 
$$\a_f:\det(V_{k+1})\ot V_{k+1}^\vee\simeq {\bigwedge}^kV_{k+1} \stackrel{{\bigwedge}^k f}{\longrightarrow}   {\bigwedge}^kV_k\simeq \det(V_k).$$
Since $\det(V_k)=\det(V_{k+1})$, we can view the dual map $\a_f^\vee$ as a map $\OO\to V_{k+1}$ and we have a well known identity $f\circ \a_f^\vee=0$
(which follows from the vanishing of ${\bigwedge}^{k+1}(f)$). Thus, 
$$\wt{\psi}(f)=\a_f^{\vee}\in H^0(C,V_{k+1})$$
is the required degree $k$ polynomial map.

To find $\wt{\phi}$, we note that it should associate to a generic $e\in \Ext^1(V_k,\OO)$ a unique (up to rescaling) $f\in \Hom(V_{k+1},V_k)$ such that $e\circ f=0$.
Now we can find an autoequivalence $\Phi$ of $D^b(C)$ such that
$$\Phi(V_{k+1})=\OO, \ \ \Phi(V_k)=V_{l+1}, \ \ \Phi(\OO[1])=V_l,$$
where $l\equiv -(k+1)^{-1} \mod (n)$, $0<l<n$. Indeed, this is essentially a statement about the action of $\SL_2(\Z)$ on the pairs of primitive vectors in $\Z^2$ (by
looking at the action of autoequivalences on pairs $(\deg,\rk)$).
To find $l\mod(n)$ one uses the invariant $\a(v_1,v_2)\in (\Z/\det(v_1,v_2))^*$ of such a pair defined by $v_1\equiv \a(v_1,v_2)\cdot v_2 \mod \det(v_1,v_2)\Z^2$
(see \cite[Sec.\ 4]{P98}).

The above autoequivalence $\Phi$ identifies the composition map 
$$\Hom(V_{k+1},V_k)\ot \Ext^1(V_k,\OO)\to \Ext^1(V_{k+1},\OO)$$ 
we are interested in with the composition map
$$\Hom(\OO,V_{l+1})\ot \Hom(V_{l+1},V_l)\to \Hom(\OO,V_l).$$
Now we can apply our construction of $\wt{\psi}$ to the latter composition map, and view it via the above isomorphisms
as the needed degree $l$ map $\wt{\phi}$ from $\Ext^1(V_k,\OO)$ to $\Hom(V_{k+1},V_k)$.

Similarly, to find $\wt{\rho}$, we find an autoequivalence $\Psi$ of $D^b(C)$ such that
$$\Psi(V_k[-1])=\OO, \ \ \Psi(\OO)=V_{m+1}, \ \ \Psi(V_{k+1})=V_m,$$
where $m\equiv -1-k^{-1} \mod (n)$, $0<m<n$. As above, this gives the required degree $m$ map $\wt{\psi}$ from $H^0(C,V_{k+1})$ to $\Ext^1(V_k,\OO)$.

\begin{rmk}\label{two-map-rem}
It is easy to see that if two sets of homogeneous polynomials of degree $d$, $(f_1,\ldots,f_n)$ and $(\wt{f}_1,\ldots,\wt{f}_n)$ induce the same birational maps 
$\P^{n-1}\dashrightarrow \P^{n-1}$ and if $(f_1,\ldots,f_n)$ have no common factor, then there exists a constant $c$ such that $\wt{f}_i=cf_i$.
Indeed, the corresponding rational function $\wt{f}_1/f_1=\ldots=\wt{f}_n/f_n$ on $\P^{n-1}$ has to be regular, hence, constant.
Since the polynomials $(p_1,\ldots,p_n)$ from Theorem B clearly have no common factor (as they define a codimension $3$ variety), it follows that in the case $k=1$,
our map $\wt{\phi}$ agrees with $(p_1,\ldots,p_n)$.
\end{rmk}

\end{document}